\documentclass[10pt,twoside,twocolumn]{article}
\usepackage{times}
\usepackage{bm}
\usepackage{graphics}
\usepackage{theorem}
\usepackage[dvipdfmx]{graphicx}
\usepackage{amsmath}
\usepackage{latexsym}
\usepackage{amssymb,mathrsfs}
\usepackage{multicol} 
\usepackage{natbib} 
\usepackage[procnumbered,ruled]{algorithm2e}
\newcounter{hdps}

\makeatletter
\newcommand{\RemoveAlgoNumber}{\renewcommand{\fnum@algocf}{\AlCapSty{\AlCapFnt\algorithmcfname}}}
\newcommand{\RevertAlgoNumber}{\algocf@resetfnum}
\makeatother
\usepackage[flushmargin]{footmisc}

\theorembodyfont{\itshape}
\newtheorem{thm}{Theorem}[section]
\newtheorem{lem}{Lemma}[section]

\newtheorem{examp}{Example}[section]

{\bf}{\rm}

\theorembodyfont{\rmfamily}
{\bf}{\rm}
{\bf}{\rm}

\newtheorem{hmalgo}{Algorithm}{\bf}{\rmfamily}
\newtheorem{hmproc}{Procedure}{\bf}{\rmfamily}
\newtheorem{rem}{Remark}[section]{\itshape}{\rmfamily}
\newenvironment{proof}{\noindent{\it Proof.~}}{\medskip}

\makeatletter
\renewcommand{\section}{%
  \@startsection{section}%
   {1}%
   {\z@}%
   {-3.5ex \@plus -1ex \@minus -.2ex}%
   {2.3ex \@plus.2ex}%
   {\normalfont\normalsize\bfseries}%
}%
\makeatother
\makeatletter
\def\eqnarray{\stepcounter{equation}\let\@currentlabel=\theequation
\global\@eqnswtrue
\global\@eqcnt\z@\tabskip\@centering\let\\=\@eqncr
$$\halign to \displaywidth\bgroup\@eqnsel\hskip\@centering
  $\displaystyle\tabskip\z@{##}$&\global\@eqcnt\@ne 
  \hfil$\;{##}\;$\hfil
  &\global\@eqcnt\tw@ $\displaystyle\tabskip\z@{##}$\hfil 
   \tabskip\@centering&\llap{##}\tabskip\z@\cr}
\makeatother

\newcommand{\vc}{\bm}

\newcommand{\ol}{\overline}

\newcommand{\wt}{\widetilde}

\newcommand{\down}[2]{\smash{\lower#1\hbox{#2}}}
\newcommand{\up}[2]{\smash{\lower-#1\hbox{#2}}}

\newcommand{\dm}{\displaystyle}

\newcommand{\PP}{\mathsf{P}}

\newcommand{\calT}{\mathcal{T}}

\newcommand{\bbA}{\mathbb{A}}
\newcommand{\bbB}{\mathbb{B}}

\newcommand{\bbF}{\mathbb{F}}

\newcommand{\bbS}{\mathbb{S}}

\newcommand{\bbZ}{\mathbb{Z}}

\newcommand{\card}{\mathrm{card}}

\newcommand{\dd}[1]{\if#11 1\!\!1 
\else {\if#1C I\!\!\!C
\else {\if#1G I\!\!\!G 
\else {\if#1J J\!\!\!J 
\else {\if#1S S\!\!\!S
\else {\if#1Z Z\!\!\!Z
\else {\if#1Q O\!\!\!\!Q
\else I\!\!#1
\fi}
\fi}
\fi}
\fi} 
\fi} 
\fi} 
\fi} 
\makeatother

\title{Binary sampling from discrete distributions}
\author{Hiroyuki Masuyama}

\begin{document}\thispagestyle{plain} 
\twocolumn[
\vspace{5mm}
\rule{0.99\textwidth}{0.4mm}
\begin{flushleft}
{\LARGE\bf
Binary sampling from discrete distributions
}

\vspace{5mm}

{\large\bf
Hiroyuki Masuyama%
}
\rule{0.99\textwidth}{0.4mm}
\vspace{20mm}

\end{flushleft}
]
\noindent
{
\textbf{Abstract}~~
This paper considers direct sampling methods from discrete target distributions. The inverse transform sampling (ITS) method is one of the most popular direct sampling methods. The main purpose of this paper is to propose a direct sampling algorithm that supersedes the binary-search ITS method (which is an improvement of the ITS method with binary search). The proposed algorithm is based on binarizing the support set of the target distribution. Thus, the proposed algorithm is referred to as binary sampling (BS). The BS algorithm consists of two procedures: backward binary sampling (BBS) and forward binary sampling (FBS). The BBS procedure draws a single sample (the first sample) from the target distribution while constructing a one-way random walk on a binary tree for the FBS procedure. By running the random walk, the FBS procedure generates the second and subsequent samples. The BBS and FBS procedures have $O(N)$ and $O(\ln N)$ time complexities, respectively, and they also have $O(N)$ space complexity, where $N+1$ is the cardinality of the support set of the target distribution. Therefore, the time and space complexities of the BS algorithm are equivalent to those of the standard (possibly best) binary-search ITS algorithm. However, the BS algorithm has two advantages over the standard binary-search ITS algorithm. First, the BBS procedure is parallelizable and thus the total running time of the BS algorithm can be reduced. Second, the BS algorithm is more accurate in terms of relative rounding error that influences generated samples.
}


\noindent
\rule{0.49\textwidth}{0.2mm}
\vspace{-10mm}
\begin{flushleft}
This research was supported in part by JSPS KAKENHI \break 
Grant Number JP15K00034.
\end{flushleft}
\vspace{-7mm}
\rule{0.49\textwidth}{0.2mm}

\smallskip

\noindent
H. Masuyama\\
Email: masuyama@sys.i.kyoto-u.ac.jp\\
{\it Department of Systems
Science, Graduate School of Informatics, Kyoto University
Kyoto 606-8501, Japan}

\smallskip

\noindent
{

\smallskip

}

\smallskip

\section{Introduction}\label{sec-intro}

In this paper, we consider sampling from discrete target (probability) distributions.
Sampling from target distributions is crucial for Monte Carlo methods. The methods of sampling can be categorized into two groups: Markov chain Monte Carlo (MCMC) methods (see, e.g., \citealt{Broo11}) and direct sampling methods (i.e., non MCMC methods; see, e.g., \citealt{Devr86}). 

MCMC methods include Metropolis-Hastings algorithm, Gibbs sampling, slice sampling, etc. Basically, MCMC methods are approximate sampling methods, except for ``Coupling From The Past (CFTP)" (see, e.g., \citealt{Hube16-book}). The CFTP algorithm achieves {\it exact sampling (or perfect sampling)}, i.e., generates samples that {\it exactly (or perfectly)} follow the target distribution. 

Direct sampling methods achieve exact sampling, and include inverse transform sampling (ITS), acceptance-rejection sampling, and importance sampling, etc. These methods are not, in general, suitable for high-dimensional target distributions. However, the methods do not have to construct Markov chains and therefore are more easily implementable than MCMC methods.

Among the above direct sampling methods, we focus on the ITS method (see, e.g., \citealt[Section III.2.1]{Devr86}). This has three reasons: (i) The ITS method is often used to generate samples from proposal distributions in acceptance-rejection sampling and importance sampling; (ii) itself does not require any proposal distribution; and (iii) is flexible and easily implementable for discrete target distributions.

It should be noted that the naive algorithm of the ITS method (called the {\it naive ITS algorithm}, for short) requires the cumulative distribution function of the target distribution in order to generate samples. Thus, if we know only the probability mass function of the target distribution, we have to compute its cumulative distribution function. This preprocessing has time complexity of $O(N)$, where $O(\cdot)$ represents Big-$O$ notation and (following the definition introduced later) $N+1$ denotes the {\it cardinality of the support set} (called {\it size} for short) of the target distribution. Furthermore, the naive ITS algorithm takes, at worst, $O(N)$ time to generate a sample by mapping a uniform random number to an element of the support set of the target distribution.

To reduce the running time of this mapping, we can use {\it binary search}. For simplicity, we call such an improvement of the ITS method with binary search the {\it binary-search ITS method}. The binary-search ITS method has some algorithms depending on what type of a binary tree is constructed. The standard (and probably best) binary-search ITS algorithm constructs a complete binary tree such that its leaves store the probabilities (masses) of the target distribution and the other nodes (the root and internal nodes) store the sums of the probabilities of the leaves retrieved sequentially by inorder traversal (see \citealt[Section III.2]{Devr86}). Although this standard binary-search ITS algorithm generates a sample in $O(\ln N)$ time, its preprocessing (constructing the binary tree) has $O(N)$ time complexity and produces $O(N)$ relative rounding error in computing the probabilities stored in the root and internal nodes.

The main contribution of this paper is to propose a direct sampling algorithm that supersedes the binary-search ITS method and, of course, the ITS method. The proposed algorithm is based on binarizing the support set of the target distribution. Hence, we refer to the proposed algorithm as {\it binary sampling (BS)}.
The BS algorithm consists of two procedures: {\it backward binary sampling (BBS)} and {\it forward binary sampling (FBS)}.
Although the BBS procedure is the preprocessing of the FBS one, the former
 generates a single sample while constructing a one-way random walk on a binary tree for the latter, which is achieved by the pairwise summation of the target distribution. By running the one-way random walk, the FBS procedure generates samples.

The BBS and FBS procedures have $O(N)$ and $O(\ln N)$ time complexities, respectively, which are equivalent to those of the preprocessing and main processing of the standard binary-search ITS algorithm. 
It should be noted that the BBS procedure (the preprocessing of the BS algorithm) generates a sample whereas the preprocessing of the standard binary-search ITS algorithm does not. 
In addition, since the BBS procedure
performs the pairwise summation of the target distribution, this procedure causes only $O(\ln N)$ relative rounding error and is parallelizable.
Therefore, our BS algorithm is more accurate and scalable than the standard binary-search ITS algorithm.

The rest of this paper is divided into four sections. Section~\ref{sec-preliminaries} presents preliminary results together with basic definitions and notation. Section~\ref{sec-main} describes the proposed algorithm, i.e., the BS algorithm. Section~\ref{sec-discussion} compares the BS algorithm with the naive ITS algorithm and the standard binary-search ITS algorithm. Finally, Section~\ref{sec-applicability} considers the adaptability of the BS algorithm to high-dimensional target distributions. 

\section{Preliminaries}\label{sec-preliminaries}

We consider sampling from a target distribution with support set $\bbZ_N := \{0,1,\dots,N\}$, where $N$ denotes a nonnegative integer, i.e., $N \in \bbZ_+:=\{0,1,2,\dots\}$.
Let $\{\pi(i);i \in \bbZ_N\}$ denote the target distribution. Note here that $\sum_{i=0}^N \pi(i) = 1$ and 
\begin{equation}
\min_{i\in\bbZ_N}\pi(i)>0.
\label{cond-pi}
\end{equation}
Furthermore, let $d$ denote an integer such that $2^{d-1} < N+1 \le 2^d$, 
or equivalently,
\begin{equation}
d = \lceil \log_2 (N+1) \rceil.
\label{defn-d-02}
\end{equation}
For convenience, we set
\[
\pi(i) = 0,\qquad i=N+1,N+2,\dots,2^d - 1.
\]

We need more definitions. For $\ell \in \bbZ_d$, let $\vc{n}^{(\ell)}$ denote
\begin{eqnarray*}
\vc{n}^{(0)} &=& \varnothing,
\\
\vc{n}^{(\ell)} &=& (n_1,n_2, \dots, n_{\ell}) \in \bbB^{\ell},
\qquad \ell=1,2,\dots,d, 
\end{eqnarray*}
where $\bbB=\{0,1\}$.
For $\ell=1,2,\dots,d$, let $\sigma_{\ell}$ denote a function from $\bbB^{\ell}$ to $\bbZ_{2^{\ell}-1}$ such that, for $\vc{n}^{(\ell)}\in\bbB^{\ell}$, 
\begin{equation*}
\sigma_{\ell}(\vc{n}^{(\ell)})
= n_1 2^0 + n_2  2^1 + \cdots + n_{\ell} 2^{\ell-1}
= \sum_{j=1}^{\ell} n_{j} 2^{j-1},
\end{equation*}
which is equivalent to the binary number $n_{\ell}n_{\ell-1}\cdots n_1$.
We then define $\{\varpi_{\ell}(\vc{n}^{({\ell})});\vc{n}^{({\ell})} \in \bbB_{\ell}\}$'s, $\ell\in\bbZ_d$, by the recursion:
\begin{eqnarray}
\varpi_d( \vc{n}^{(d)} ) 
&=& \pi( \sigma_d(  \vc{n}^{(d)} ) )
\nonumber
\\
&=& \pi(\mbox{$\sum_{j=1}^d n_{j} 2^{j-1}$}), \qquad \vc{n}^{(d)} \in \bbB^d,
\label{defn-pi_d(n)}
\end{eqnarray}
and, for $\ell=d-1,d-2,\dots,0$,
\begin{eqnarray}
\varpi_{\ell}(\vc{n}^{(\ell)}) 
&=& \varpi_{\ell+1}(\vc{n}^{(\ell)},0) 
\nonumber
\\
&& {} ~ 
 + \varpi_{\ell+1}(\vc{n}^{(\ell)},1),
\qquad \vc{n}^{(\ell)} \in \bbB^{\ell}.
\label{defn-pi_i(n)}
\end{eqnarray}
Note here that the computation of $\{\varpi_{\ell}(\vc{n}^{({\ell})}); \vc{n}^{({\ell})} \in \bbB_{\ell}\}$'s, $\ell\in\bbZ_d$, is the pairwise summation of the target distribution (see Fig.~\ref{fig-varpi}).
\begin{figure*}[htb]
	\centering
	\includegraphics[scale=0.33,bb=0 0 698 264]{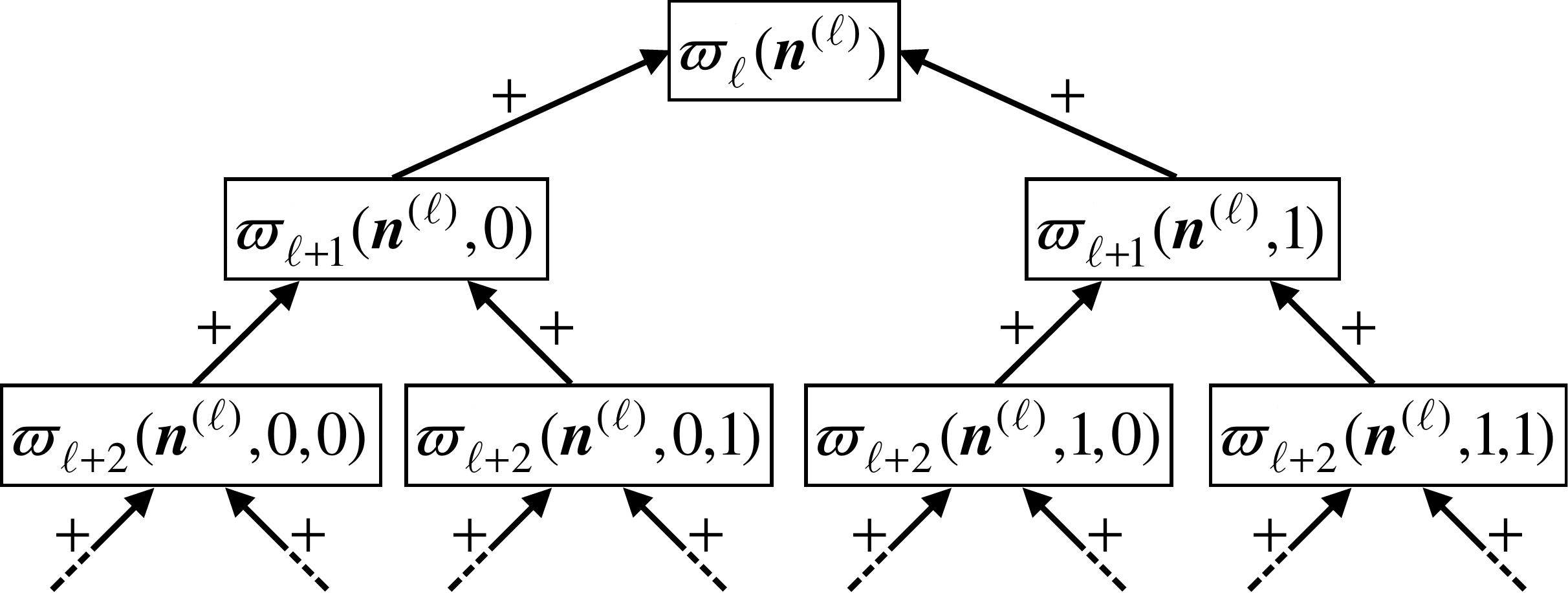}
	\caption{Computation of $\{\varpi_{\ell}(\vc{n}^{({\ell})})\}$ by pairwise summation}
	\label{fig-varpi}
\end{figure*}

It follows from (\ref{defn-pi_d(n)}) and (\ref{defn-pi_i(n)}) that, for $\ell \in \bbZ_{d-1}$,
\begin{eqnarray}
\varpi_{\ell}(\vc{n}^{(\ell)}) 
&=& 
\sum_{ n_{\ell+1} \in\bbB  } 
\sum_{ n_{\ell+2} \in\bbB  } 
\cdots 
\sum_{ n_d \in\bbB  } 
\varpi_d(\vc{n}^{(d)}) 
\nonumber
\\
&=& 
\sum_{ n_{\ell+1} \in\bbB  } 
\sum_{ n_{\ell+2} \in\bbB  } 
\cdots 
\sum_{ n_d \in\bbB  } 
\pi( \sigma_d(  \vc{n}^{(d)} ) ), 
\label{eqn-varpi_{ell}(n^{(ell)})} 
\end{eqnarray}
where
\begin{equation}
\varpi_0(\vc{n}^{(0)}) = 1.
\label{eqn-varpi_0}
\end{equation}
It also follows from (\ref{cond-pi}) and (\ref{defn-pi_d(n)}) that $\sigma_d(  \vc{n}^{(d)} ) \in \bbZ_N$ if and only if $\varpi_d(\vc{n}^{(d)}) > 0$. For later use, let $\bbB_+^{(d)} $ denote
\begin{eqnarray}
\bbB_+^{(d)} 
&=& \{\vc{n}^{(d)} \in \bbB^d: \sigma_d(  \vc{n}^{(d)} ) \in \bbZ_N\}
\nonumber
\\
&=& \{\vc{n}^{(d)} \in \bbB^d: \varpi_d(\vc{n}^{(d)}) > 0\}.
\label{defn-B_+^{(d)}}
\end{eqnarray}
For $\ell=1,2,\dots,d-1$, let $\bbB_+^{(\ell)}$ denote
\begin{eqnarray}
\bbB_+^{(\ell)} 
&=& 
\{
\vc{n}^{(\ell)} \in \bbB^{\ell}: 
(\vc{n}^{(\ell)},n_{\ell+1}, \dots,n_d)\in \bbB_+^{(d)} 
\}
\nonumber
\\
&=&
\{
\vc{n}^{(\ell)} \in \bbB^{\ell}: 
\sigma_d(  \vc{n}^{(d)} ) \in \bbZ_N
\},
\label{defn-B_+^{(l)}}
\end{eqnarray}
where the second equality holds due to (\ref{defn-B_+^{(d)}}) and $\vc{n}^{(d)} = (\vc{n}^{(\ell)},n_{\ell+1}, \dots,n_d)$. Equations (\ref{eqn-varpi_{ell}(n^{(ell)})}), (\ref{defn-B_+^{(d)}}) and (\ref{defn-B_+^{(l)}}) imply that
\begin{eqnarray}
\vc{n}^{(d)} \in \bbB_+^{(d)}
&\Longrightarrow&
\vc{n}^{(\ell)} \in \bbB_+^{(\ell)} ~\mbox{and}~ \varpi_{\ell}(\vc{n}^{(\ell)}) > 0
\nonumber
\\&& \mbox{for all}~\ell=1,2,\dots,d-1.
\label{eqn-161007-02}
\end{eqnarray}

We now prove a lemma, which presents a basic idea behind our sampling algorithm.
\begin{lem}\label{lem-pi(sigma_d(n))}
For $\vc{n}^{(d)} \in \bbB_+^{(d)}$,
\begin{align}
\pi( \sigma_d( \vc{n}^{(d)} ) )
&= 
\prod_{\ell=0}^{d-1}
 \{ \rho_{\ell}(\vc{n}^{(\ell)}) \}^{n_{\ell+1}}  
\nonumber
\\
& {} \qquad \times 
 \{ \ol{\rho}_{\ell}(\vc{n}^{(\ell)}) \}^{1 - n_{\ell+1}},
\label{eqn-varpi(n)-02}
\end{align}
where $\rho_{\ell}(\vc{n}^{(\ell)})$'s and $\ol{\rho}_{\ell}(\vc{n}^{(\ell)})$'s, $\ell \in \bbZ_{d-1}$, $\vc{n}^{(\ell)} \in \bbB_+^{(\ell)}$, are given by
\begin{eqnarray}
\rho_{\ell}(\vc{n}^{(\ell)}) 
&=& 
{\varpi_{\ell+1}(\vc{n}^{(\ell)},1) 
\over \varpi_{\ell}(\vc{n}^{(\ell)})
}
\nonumber
\\
&=& 
{\varpi_{\ell+1}(\vc{n}^{(\ell)},1) 
\over \varpi_{\ell+1}(\vc{n}^{(\ell)},0) + \varpi_{\ell+1}(\vc{n}^{(\ell)},1)},
\label{defn-rho_l(n)}
\\
\ol{\rho}_{\ell}(\vc{n}^{(\ell)}) &=& 1 - \rho_{\ell}(\vc{n}^{(\ell)}).
\label{defn-ol{rho}_l(n)}
\end{eqnarray}
\end{lem}

\begin{proof}
Fix $\vc{n}^{(d)} \in \bbB_+^{(d)}$ arbitrarily.
It then follows from (\ref{eqn-161007-02}), (\ref{defn-rho_l(n)}) and (\ref{defn-ol{rho}_l(n)}) that $\rho_{\ell}(\vc{n}^{(\ell)})$'s and $\ol{\rho}_{\ell}(\vc{n}^{(\ell)})$'s are well-defined for $\ell \in \bbZ_{d-1}$ and $\vc{n}^{(\ell)} \in \bbB_+^{(\ell)}$. Note here that the second equality of (\ref{defn-rho_l(n)}) holds due to (\ref{defn-pi_i(n)}). 
Note also that (\ref{defn-rho_l(n)}) and (\ref{defn-ol{rho}_l(n)}) yield
\[
\ol{\rho}_{\ell}(\vc{n}^{(\ell)})
= {\varpi_{\ell+1}(\vc{n}^{(\ell)},0) 
\over \varpi_{\ell}(\vc{n}^{(\ell)})
},
\]
and thus
\begin{eqnarray}
{\varpi_{\ell+1}(\vc{n}^{(\ell+1)}) 
\over \varpi_{\ell}(\vc{n}^{(\ell)})
}
&=& {\varpi_{\ell+1}(\vc{n}^{(\ell)},n_{\ell+1}) 
\over \varpi_{\ell}(\vc{n}^{(\ell)})
}
\nonumber
\\
&=& \{ \rho_{\ell}(\vc{n}^{(\ell)}) \}^{n_{\ell+1}}  
 \{ \ol{\rho}_{\ell}(\vc{n}^{(\ell)}) \}^{1 - n_{\ell+1}}. \qquad
\label{eqn-161007-01}
\end{eqnarray}
From (\ref{defn-pi_d(n)}), (\ref{eqn-varpi_0}) and (\ref{eqn-161007-01}), we have
\begin{eqnarray*}
\lefteqn{
\pi( \sigma_d( \vc{n}^{(d)} ) )
}
\quad &&
\nonumber
\\
&=& \varpi_0( \vc{n}^{(0)} ) 
{\varpi_{1}(\vc{n}^{(1)})
\over \varpi_0(\vc{n}^{(0)})
}
{\varpi_{}(\vc{n}^{(2)}) 
\over \varpi_{1}(\vc{n}^{(1)})
}
\cdots
{\varpi_d(\vc{n}^{(d)}) 
\over \varpi_{d-1}(\vc{n}^{(d-1)})
}
\nonumber
\\
&=& 
\prod_{\ell=0}^{d-1}
 \{ \rho_{\ell}(\vc{n}^{(\ell)}) \}^{n_{\ell+1}}  
 \{ \ol{\rho}_{\ell}(\vc{n}^{(\ell)}) \}^{1 - n_{\ell+1}},
\end{eqnarray*}
which shows that (\ref{eqn-varpi(n)-02}) holds. 
\end{proof}

\section{The proposed sampling algorithm: binary sampling}\label{sec-main}

In this section, we describe our sampling algorithm. As mentioned in Section~\ref{sec-intro}, the algorithm consists of the two procedures:  backward binary sampling
(BBS) and forward binary sampling (FBS). 
In what follows, we provide the details of the BBS and FBS procedures.

The BBS procedure is the preprocessing of the FBS procedure. The BBS procedure computes the probabilities $\{\varpi_{\ell}(\vc{n}^{(\ell)});\vc{n}^{(\ell)} \in \bbB_+^{(\ell)}\}$ for $\ell=d-1,d-2,\dots,1$ by the pairwise summation of $\{\varpi_d(\vc{n}^{(d)});\vc{n}^{(d)} \in \bbB_+^{(d)}\}$. Using the computed probabilities, the BBS procedure constructs a one-way random walk on a binary tree, which is used by the FBS procedure.

To describe this one-way random walk, we introduce some definitions. Let $\bbS$ denote
\[
\bbS = \bigcup_{\ell=0}^d \bbB_+^{(\ell)},
\]
where $\bbB_+^{(0)} = \varnothing$.
Let $\{X_{\ell};\ell \in \bbZ_d\}$ denote a random walk with state space $\bbS$, which evolves in the following law:
\begin{equation}
\PP(X_0 = \varnothing) = 1,
\label{initial-cond}
\end{equation}
and, for $\ell \in \bbZ_{d-1}$, $\vc{n}^{(\ell)} \in \bbB_+^{(\ell)}$ and $n_{\ell+1} \in \bbB$,
\begin{eqnarray}
\lefteqn{
\PP(X_{\ell+1} = (\vc{n}^{(\ell)},n_{\ell+1}) \mid X_{\ell} 
= \vc{n}^{(\ell)})
}
\qquad &&
\nonumber
\\
&=& \{ \rho_{\ell}(\vc{n}^{(\ell)}) \}^{n_{\ell+1}}  
 \{ \ol{\rho}_{\ell}(\vc{n}^{(\ell)}) \}^{1 - n_{\ell+1}},
\label{transition-law-RW}
\end{eqnarray}
where $\vc{n}^{(\ell+1)} = (\vc{n}^{(\ell)},n_{\ell+1})$, and where $\rho_{\ell}(\vc{n}^{(\ell)})$ and $\ol{\rho}_{\ell}(\vc{n}^{(\ell)})$ are easily calculated by (\ref{defn-rho_l(n)}) and (\ref{defn-ol{rho}_l(n)}) with $\varpi_{\ell+1}(\vc{n}^{(\ell)},0)$ and $\varpi_{\ell+1}(\vc{n}^{(\ell)},1)$. 

The state space $\bbS$ of the one-way random walk $\{X_{\ell}\}$ is considered a binary tree such that the root is labeled with $\vc{n}^{(0)} = \varnothing$ and each of all the other nodes has a label in binary vector form that consists of its parent label and a binary digit, where ``0" and ``1" corresponds to left and right children, respectively. For example, the left and right children (if any) of node $\vc{n}^{(\ell)}$ are labeled with $(\vc{n}^{(\ell)},0)$ and $(\vc{n}^{(\ell)},1)$, respectively (see Fig.~\ref{fig-FBS}).
\begin{figure*}[htb]
	\centering
	\includegraphics[scale=0.44,bb=0 0 961 308]{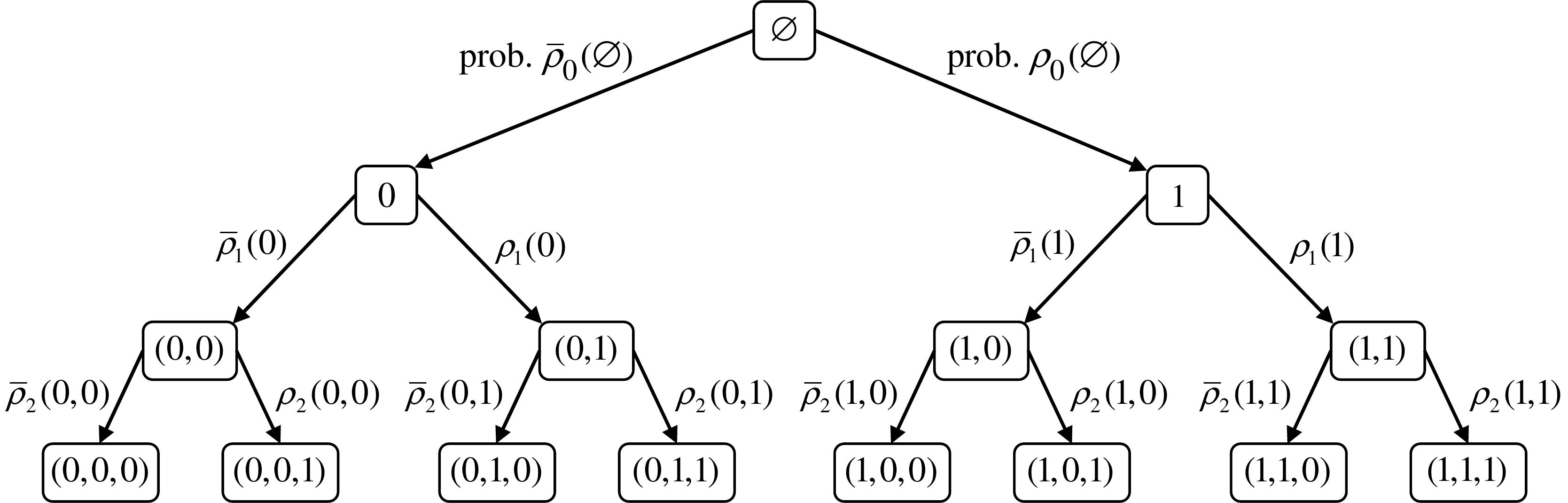}
	
	\caption{One-way random walk on a binary tree for the FBS procedure ($N = 2^3 - 1 = 7$)}
	\label{fig-FBS}
\end{figure*}
In this perspective, the one-way random walk $\{X_{\ell}\}$ starts from the root of the binary tree, moves down according to the transition probabilities $\rho_{\ell}(\vc{n}^{(\ell)})$'s and $\ol{\rho}_{\ell}(\vc{n}^{(\ell)})$'s and ends at one of the leaves. 

From Lemma~\ref{lem-pi(sigma_d(n))}, we have the following result.
\begin{lem}\label{lem-forward}
\begin{equation*}
\PP(X_d = \vc{n}^{(d)}) = \pi(\sigma_d(\vc{n}^{(d)})),\qquad \vc{n}^{(d)} \in \bbB_+^{(d)}.
\end{equation*}
\end{lem}

\begin{proof}
It follows from (\ref{initial-cond}), (\ref{transition-law-RW}) and Lemma~\ref{lem-pi(sigma_d(n))} that, for $\vc{n}^{(d)} \in \bbB_+^{(d)}$,
\begin{eqnarray*}
\PP(X_d = \vc{n}^{(d)})
&=& \prod_{\ell=0}^{d-1}
\{ \rho_{\ell}(\vc{n}^{(\ell)}) \}^{n_{\ell+1}}  
 \{ \ol{\rho}_{\ell}(\vc{n}^{(\ell)}) \}^{1 - n_{\ell+1}}
\nonumber
\\
&=& \pi(\sigma_d(\vc{n}^{(d)})),
\end{eqnarray*}
which completes the proof. 
\end{proof}

Lemma~\ref{lem-forward} implies that the one-way random walk $\{X_{\ell};\ell\in\bbZ_d\}$ generates samples following the target distribution. Indeed, the FBS procedure achieves such sampling by choosing the values of $n_{\ell}$'s, $\ell=1,2,\dots,d$, in
the {\it forward} order, i.e., in the order of $n_1,n_2,\dots,n_d$.
The way of choosing the $n_{\ell}$'s is such that
\begin{eqnarray}
n_{\ell}
=
\left\{
\begin{array}{ll}
1, & \quad \mbox{with prob.~} 
\rho_{\ell-1}( \vc{n}^{(\ell-1)} ),
\\
0, & \quad \mbox{with prob.~} \ol{\rho}_{\ell-1}(\vc{n}^{(\ell-1)}),
\end{array}
\right.
\label{choosing-n_l}
\end{eqnarray}
where $\ol{\rho}_{\ell-1}(\vc{n}^{(\ell-1)}) =1 - \rho_{\ell-1}( \vc{n}^{(\ell-1)} )$ and
\begin{equation}
\rho_{\ell-1}( \vc{n}^{(\ell-1)} ) 
= {\varpi_{\ell}( \vc{n}^{(\ell-1)},1 ) 
\over 
\varpi_{\ell}( \vc{n}^{(\ell-1)}, 0 ) +
\varpi_{\ell}( \vc{n}^{(\ell-1)}, 1 )
}.
\label{eqn-rho_{l-1}(n)}
\end{equation}
The obtained vector $\vc{n}^{(d)} = (n_1,n_2,\dots,n_d)$ is converted to the integer $i_{\ast} = \sigma_d(\vc{n}^{(d)})$, which is a sample from the target distribution. The description of the FBS procedure is summarized in 
Procedure~\ref{proc-FBS}.

\begin{hmproc}[FBS: Forward binary sampling]\label{proc-FBS}
\hfill 

\noindent
{
{\bf Input:}
\begin{minipage}[t]{0.8\textwidth}
$\{\varpi_{\ell}(\vc{n}^{(\ell)});\ell=1,2,\dots,d, \vc{n}^{(\ell)} \in \bbB_+^{(\ell)}\}$
\end{minipage}

\smallskip

\noindent
{\bf Output:} Sample $i_{\ast} \in \bbZ_N$ from $\{\pi(i)\}$

\vspace{-1mm}

\begin{enumerate}
\setlength{\parskip}{0cm}
\setlength{\itemsep}{0cm}
\item For $\ell=1,2,\dots,d$, choose $n_{\ell} \in \bbB$  by (\ref{choosing-n_l}).
\item Return $i_{\ast} = \sigma_d(\vc{n}^{(d)}) \in \bbZ_N$.
\end{enumerate}
}
\end{hmproc}

The following theorem is an immediate consequence of Lemma~\ref{lem-forward} and Procedure~\ref{proc-FBS}. Thus, we omit its proof. 

\begin{thm}
The FBS procedure generates samples following the target distribution $\{\pi(i);i\in\bbZ_N\}$. 
\end{thm}

The FBS procedure uses the one-way random walk $\{X_{\ell};\ell\in\bbZ_d\}$, which is constructed by the BBS procedure. It is remarkable that the BBS procedure not only constructs this random walk but also draws a sample from the target distribution $\{\pi(i)\}$. More specifically, the BBS procedure draws a binary vector $(n_1,n_2\dots,n_d)$ from $\bbB_+^{(d)}$ by choosing the values of the $n_{\ell}$'s by (\ref{choosing-n_l}) in the {\it backward order}, i.e., the order of $n_d,n_{d-1},\dots,n_1$. The description of the BBS procedure is summarized in Procedure~\ref{proc-BBS}. In addition, Fig.~\ref{fig-BBS} provides an example of the behavior of the BBS procedure.
\begin{hmproc}[BBS: Backward binary sampling]\label{proc-BBS}
\hfill \smallskip

\noindent
{
{\bf Input:}
Target distribution $\{\pi(i);i\in\bbZ_N\}$

\smallskip

\noindent
{\bf Output:} 
\begin{minipage}[t]{0.7\textwidth}
Sample $i_{\ast} \in \bbZ_N$ from $\{\pi(i);i\in\bbZ_N\}$ and \\
$\{\varpi_{\ell}(\vc{n}^{(\ell)});\ell=1,2,\dots,d, \vc{n}^{(\ell)} \in \bbB_+^{(\ell)}\}$
\end{minipage}

\vspace{-1mm}

\begin{enumerate}
\setlength{\parskip}{0cm}
\setlength{\itemsep}{0cm}
\item Set $\bbA = \bbB_+^{(d)}$.
\item For $\ell=d,d-1,\dots,1$, execute the following iteration.
\begin{description}
\setlength{\parskip}{0cm}
\setlength{\itemsep}{0cm}
\item[{\bf Iteration}:] For each $\vc{n}^{(\ell-1)} \in \bbB_+^{(\ell-1)}$, perform Steps (a)--(c):
\end{description}
\vspace{-2mm}
\begin{enumerate}
\setlength{\parskip}{0cm}
\setlength{\itemsep}{0cm}
\item If $\ell=d$, store the probabilities:
\[
\varpi_d(\vc{n}^{(d-1)},n_d)
= \pi(\sigma(\vc{n}^{(d-1)},n_d)),\quad n_d \in \bbB;
\]
otherwise (i.e., if $\ell \le d-1$) compute the probabilities $\varpi_{\ell}(\vc{n}^{(\ell-1)},n_{\ell})$'s $n_{\ell} \in \bbB$, by (\ref{defn-pi_i(n)}) and store the results;
\item choose the value of $n_{\ell}$ by (\ref{choosing-n_l}); and
\item delete the vector $( \vc{n}^{(\ell-1)},1-n_{\ell},n_{\ell+1},\dots,n_d)$ from~$\bbA$.
\end{enumerate}
\item Return $i_{\ast} = \sigma_d(\vc{n}^{(d)}) \in \bbZ_N$ with the (unique) element $\vc{n}^{(d)}$ of $\bbA$.
\end{enumerate}
}
\end{hmproc}

\medskip

\begin{figure*}[htb]
	\centering
	\includegraphics[scale=0.4,bb=0 0 897 421]{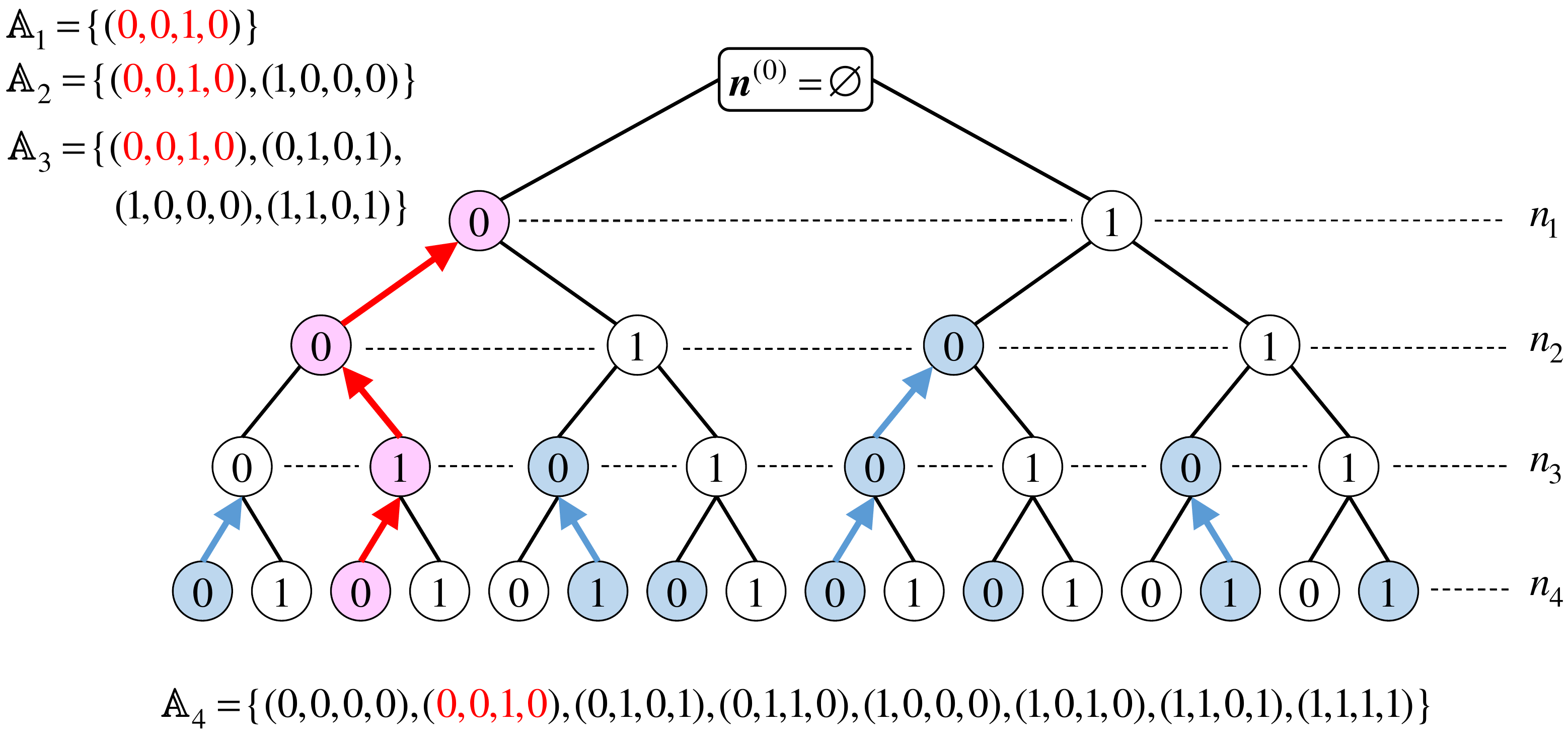}
	\caption{Example of the behavior of the BBS procedure ($N = 2^4 - 1=15$)}
	\label{fig-BBS}
\end{figure*}

The following theorem guarantees that the BBS procedure (i.e., Procedure~\ref{proc-BBS}) works well. The proof of this theorem is given in Appendix.
\begin{thm}\label{thm-BBS}
Steps (i) and (ii) of the BBS procedure result in set $\bbA$ consisting of only one element. Furthermore, Step (iii) of the BBS procedure returns $i_{\ast} \in \bbZ_N$ with probability $\pi(i_{\ast})$.
\end{thm}

\begin{rem}
In each iteration of Step (ii), the BBS procedure selects, by {\it coin toss (appropriately biased in each selection)},
the candidates of the binary expression of a possible sample from the target distribution (see Fig.~\ref{fig-BBS}).
Theorem~\ref{thm-BBS} implies that a desired result is what goes through the whole process of selections whatever it is. Therefore, the coin tosses within one iteration need not be independent, though those between different iterations must be independent.
\end{rem}

\begin{rem}\label{rem-backward-BS-02}
The BBS procedure works well even though the target distribution $\{\pi(i);i\in\bbZ_N\}$ is not normalized, that is, $\pi(i)$ is expressed as
\begin{equation}
\pi(i) = \kappa \wt{\pi}(i),\qquad i \in \bbZ_N,
\label{eqn-pi(i)-unnormalized}
\end{equation}
where $\kappa$ is an unknown positive constant and $\{\wt{\pi}(i);i\in\bbZ_N\}$ is a given sequence of positive numbers. In such a case, we define
\begin{equation*}
\wt{\varpi}_d(\vc{n}^{(d)}) 
= \wt{\pi}( \sigma_d(  \vc{n}^{(d)} ) ), 
\qquad \vc{n}^{(d)} \in \bbB_+^{(d)},
\end{equation*}
and compute, for $\ell\in\bbZ_{d-1}$ and $\vc{n}^{(\ell)} \in \bbB_+^{(\ell)}$,
\begin{eqnarray}
\wt{\varpi}_{\ell}(\vc{n}^{(\ell)}) 
&=& 
\sum_{ n_{\ell+1} \in\bbB  } 
\sum_{ n_{\ell+2} \in\bbB  } 
\cdots 
\sum_{ n_d \in\bbB  } 
\wt{\pi}( \sigma_d(  \vc{n}^{(d)} ) ), \qquad
\label{eqn-wt{varpi}_l(n)}
\end{eqnarray}
by the recursion (\ref{defn-pi_d(n)}) and (\ref{defn-pi_i(n)}) with the $\varpi_{\ell}(\vc{n}^{(\ell)})$'s replaced by the $\wt{\varpi}_{\ell}(\vc{n}^{(\ell)})$'s. Note that 
\begin{equation}
\wt{\varpi}_{\ell}(\vc{n}^{(\ell)}) = \varpi_{\ell}(\vc{n}^{(\ell)})/\kappa,
~~~ \ell\in\bbZ_{d-1},\,\vc{n}^{(\ell)} \in \bbB_+^{(\ell)},
\label{eqn-170416-01}
\end{equation}
which follows from (\ref{eqn-varpi_{ell}(n^{(ell)})}), (\ref{eqn-pi(i)-unnormalized}) and (\ref{eqn-wt{varpi}_l(n)}). We then calculate, for $\ell\in\bbZ_{d-1}$ and $\vc{n}^{(\ell)} \in \bbB_+^{(j)}$, 
\begin{eqnarray}
\wt{\rho}_{\ell}(\vc{n}^{(\ell)}) 
&:=& 
{\wt{\varpi}_{\ell+1}(\vc{n}^{(\ell)},1) 
\over \wt{\varpi}_{\ell}(\vc{n}^{(\ell)})
}
\nonumber
\\
&=&
{ \wt{\varpi}_{\ell+1}(\vc{n}^{(\ell)},1) 
\over \wt{\varpi}_{\ell+1}(\vc{n}^{(\ell)},0) 
+ \wt{\varpi}_{\ell+1}(\vc{n}^{(\ell)},1)
}.
\label{defn-wt{rho}_l(n)}
\end{eqnarray}
Equations (\ref{defn-rho_l(n)}), (\ref{eqn-170416-01}) and (\ref{defn-wt{rho}_l(n)}) show that $\wt{\rho}_{\ell}(\vc{n}^{(\ell)}) = \rho_{\ell}(\vc{n}^{(\ell)})$ for all $\ell\in\bbZ_{d-1}$ and $\vc{n}^{(\ell)} \in \bbB_+^{(\ell)}$. 
\end{rem}

We are now ready to describe our BS algorithm, which
is summarized in Algorithm~\ref{algo-BS}.

\begin{hmalgo}[BS: Binary sampling]\label{algo-BS}
\hfill \smallskip

\noindent
{
{\bf Input:} Target distribution $\{\pi(i);i\in\bbZ_N\}$\\
{\bf Output:} Samples from $\{\pi(i);i\in\bbZ_N\}$

\vspace{-1mm}

\begin{enumerate}
\setlength{\parskip}{0cm}
\setlength{\itemsep}{0cm}
\item Perform the BBS procedure (Procedure~\ref{proc-BBS}) once, which generates a sample.
\item Repeat the FBS procedure (Procedure~\ref{proc-FBS}) as many times as necessary.
\item Return the generated samples.
\end{enumerate}
}
\end{hmalgo}

\medskip

In the rest of this section, we discuss the performance of Algorithm~\ref{algo-BS}. Clearly, the time complexity of the BBS procedure is dominated by Step (ii) of Procedure~\ref{proc-BBS}, where the probabilities
\begin{equation}
\{\varpi_{\ell}(\vc{n}^{(\ell)});\ell=1,2,\dots,d,\vc{n}^{(\ell)} \in \bbB_+^{(\ell)} \}
\label{eqn-prob-varpi_l}
\end{equation}
are computed (if necessarily) and stored. The total number of these probabilities is $O(N)$. Thus, the time and space complexities of the BBS procedure are $O(N)$. It should be noted that the probabilities $\{\varpi_{\ell}(\vc{n}^{(\ell)}) \}$ in (\ref{eqn-prob-varpi_l}) are computed by pairwise summation. Therefore, the computation of theses probabilities is parallelizable, and the results include only $O(\ln N)$ relative rounding error (see, e.g., \citealt{High93}).
On the other hand, the FBS procedure determines the values of $n_1,n_2,\dots,n_d$ by running the one-way random walk $\{X_{\ell};\ell\in\bbZ_d\}$. Thus, the FBS procedure has time complexity of
\[
O(d) = O(\log_2 N) = O(\ln N),
\]
where the first equality follows from (\ref{defn-d-02}). The FBS procedure also has $O(N)$ space complexity for the probabilities $\{\varpi_{\ell}(\vc{n}^{(\ell)})\}$ in (\ref{eqn-prob-varpi_l}) and the values of $n_1,n_2,\dots,n_d$. 

As a result, the BS algorithm has $O(N)$ time and space complexities, though this algorithm generates the first sample in $O(N)$ time and the second and subsequent samples in $O(\ln N)$ time. The obtained samples are influenced by $O(\ln N)$ relative rounding error. Finally, Table~\ref{table-BS} summarizes the performance of the BS algorithm.

\begin{table}[htb]
\centering
\caption{Performance of BS algorithm}\label{table-BS}
\begin{tabular}{|r||c|c|}
\hline
				 & \multicolumn{2}{c|}{BS}	
\\
\cline{2-3}
				 &	BBS   & FBS
\\
\hline
\hline
Time complexity  &  $O(N)$ & $O(\ln N)$
\\
\hline
Space complexity &  \multicolumn{2}{c|}{$O(N)$}
\\
\hline
Relative rounding error   &  \multicolumn{2}{c|}{$O(\ln N)$} \\ 
\hline
\end{tabular}
\end{table}

\section{Comparison with ITS algorithms}\label{sec-discussion}

In this section, we compare our BS algorithm with two ITS algorithms: (a) the naive ITS algorithm; and (b) the standard binary-search ITS algorithm. To this end, we define $\{\ol{\pi}(i);i\in\bbZ_N\}$ as the cumulative distribution function of the target distribution $\{\pi(i);i\in\bbZ_N\}$, i.e.,
\[
\ol{\pi}(i) = \sum_{k=0}^i \pi(k), \qquad i \in \bbZ_N.
\]
We then assume that no explicit expressions of the cumulative distribution function $\{\ol{\pi}(i)\}$ are given, which implies that we have to compute $\{\ol{\pi}(i)\}$ or its equivalent, in order to perform the ITS method.

\subsection{Naive ITS}

We begin with the description of the naive ITS algorithm, which is summarized in Algorithm~\ref{algo-ITS}. 

\begin{hmalgo}[Naive ITS]\label{algo-ITS}
\hfill \smallskip

\noindent
{
{\bf Input:}  Target distribution $\{\pi(i);i\in\bbZ_N\}$\\
{\bf Output:} Sample $i_{\ast}$ from $\{\pi(i);i\in\bbZ_N\}$

\vspace{-1mm}

\begin{enumerate}
\setlength{\parskip}{0cm}
\setlength{\itemsep}{0cm}
\item Set $\ol{\pi}(0) = \pi(0)$, and, for $k=1,2,\dots,N$, compute and store $\ol{\pi}(k) = \ol{\pi}(k-1) + \pi(k)$.
\item Generate a uniform random number $u$ in $(0,1)$.
\item Return $i_{\ast} \in \bbZ_N$ such that $\ol{\pi}(i_{\ast}-1) < u \le \ol{\pi}(i_{\ast})$, where $\ol{\pi}(-1) = 0$.
\end{enumerate}
}
\end{hmalgo}

\medskip

Step (i) of Algorithm~\ref{algo-ITS} is the preprocessing step of the naive ITS algorithm, which has $O(N)$ time and space complexities, and produces $O(N)$ relative rounding error. 
Step (iii) of Algorithm~\ref{algo-ITS} is the main processing of the naive ITS algorithm, which is equivalent to identifying $i_{\ast} \in \bbZ_N$ such that
\begin{equation}
i_{\ast} = \max\{i\in\bbZ_N: u \le \ol{\pi}(i)\}.
\label{eqn-i_*}
\end{equation}
Therefore, the average time complexity of the main processing, denoted by $C_{\rm ITS}$, is given by
\begin{equation}
C_{\rm ITS}
= \sum_{i=0}^N (i+1) \pi(i)
= 1 + \mu(N),
\label{eqn-C_ITS}
\end{equation}
where $\mu(N) =\sum_{i=1}^N i \pi(i)$. By definition, $0 \le \mu(N) \le N$ and thus $1 \le C_{\rm ITS} \le N+1$.
Indeed, Examples~\ref{examp-artificial}--\ref{examp-binomial} below show that $C_{\rm ITS}$ ranges from $O(1)$ to $O(N)$.

\begin{examp}\label{examp-artificial}
Suppose that
\[
\pi(i) = 
\left\{
\begin{array}{ll}
1 - \dm{2\varepsilon \over N+1}, & \quad i = 0,
\\
\rule{0mm}{7mm}\dm{2\varepsilon \over N(N+1)}, & \quad i=1,2,\dots,N,
\end{array}
\right.
\]
with $0 < \varepsilon <(N+1)/2$. In this artificial case, $\mu(N) = \varepsilon$ and thus
\[
C_{\rm ITS} = 1 + \varepsilon = O(1).
\]
Furthermore, $C_{\rm ITS} \downarrow 1$ as $\varepsilon \downarrow 0$.
\end{examp}

\begin{examp}\label{examp-Zipf}
Suppose that $\{\pi(i);i\in\bbZ_N\}$ is a Zipf distribution with index $s > 2$, i.e.,
\[
\pi(i) = {(i+1)^{-s} \over \sum_{k=0}^N (k+1)^{-s}}, \qquad i \in \bbZ_N.
\]
We then have
\[
C_{\rm ITS}
= {\sum_{i=1}^N (i+1)^{-s+1} \over \sum_{k=0}^N (k+1)^{-s}},
\]
which leads to
\[
\lim_{N\to\infty}C_{\rm ITS} 
= {\zeta(s-1) \over \zeta(s)},
\]
where $\zeta(\cdot)$ is the Riemann zeta function. Therefore, $C_{\rm ITS} = O(1)$.
\end{examp}

\begin{examp}\label{examp-binomial}
Suppose that $\{\pi(i);i\in\bbZ_N\}$ is a binomial distribution with parameter $\gamma \in (0,1)$, i.e.,
\[
\pi(i) = {N! \over i!(N-i!)} \gamma^i (1 - \gamma)^{N-i}, \qquad i \in \bbZ_N.
\]
We then have $\mu(N) = \gamma N$ and thus $C_{\rm ITS} = O(N)$.
\end{examp}

Based on the above discussion, the performance of the naive ITS algorithm is summarized in Table~\ref{table-ITS}.
\begin{table*}[htb]
\centering
\caption{Performance of the native ITS algorithm}\label{table-ITS}
\begin{tabular}{|r||c|}
\hline
\down{2.5mm}{Average time complexity} &  $\mu(N)+1$ \\
& $O(1)$ at best; $O(N)$  at worst
\\
\hline
Space complexity &  $O(N)$
\\
\hline
Relative rounding error   & $O(N)$     \\
\hline
\end{tabular}
\end{table*}
Tables~\ref{table-BS} and \ref{table-ITS} show that our BS algorithm has space complexity of the same order as that of the naive ITS algorithm.
The tables also show that
our BS algorithm
is much more accurate than the naive ITS algorithm in terms of relative rounding error in cumulating the target distribution. 
As for the efficiency of generating samples, the BBS procedure of our BS algorithm generates a sample whereas its counterpart of the naive ITS algorithm (i.e., Step (i) of Algorithm~\ref{algo-ITS}) does not. In addition, the FBS procedure of the BS algorithm generates the second and subsequent samples in $O(\ln N)$ time. Therefore, the BS algorithm generally achieves high performance. On the other hand, the naive ITS algorithm can achieve extremely high performance in some cases, such as Example~\ref{examp-artificial}. 

Examples~\ref{examp-artificial}--\ref{examp-binomial} imply that {\it nonincreasing} $\{\pi(i)\}$ is basically convenient for the naive ITS algorithm. 
We now consider the suitability of the naive ITS algorithm for {\it nondecreasing} target distributions. For this purpose, we suppose that
\[
\pi(i) = {(N-i+1)^{-s} \over \sum_{k=0}^N (k+1)^{-s}}, \qquad i \in \bbZ_N,
\]
where $\{\pi(i)\}$ is nondecreasing. In this case,
\begin{equation}
\mu(N) 
= N + 1 - {\sum_{i=1}^N (i+1)^{-s+1} \over \sum_{k=0}^N (k+1)^{-s}}.
\label{eqn-mu(N)-Zipf-conversely}
\end{equation}
Substituting (\ref{eqn-mu(N)-Zipf-conversely}) into (\ref{eqn-C_ITS}) yields $C_{\rm ITS} = O(\mu(N)) = O(N)$. Thus, it may seem that nondecreasing $\{\pi(i)\}$ is inconvenient for the naive ITS algorithm. In fact, this is not necessarily the case. It should be noted that (\ref{eqn-i_*}) is equivalent to
\begin{equation}
i_{\ast} = \min\{i\in\bbZ_N: u > \ol{\pi}(i-1)\}.
\label{eqn-i_*-02}
\end{equation}
Using (\ref{eqn-i_*-02}), we can perform Step (iii) of Algorithm~\ref{algo-ITS}, whose time complexity is given by
\begin{equation}
\wt{C}_{\rm ITS}
:= \sum_{i=0}^N (N+1-i) \pi(i)
= N + 1 - \mu(N).
\label{eqn-wt{C}_ITS}
\end{equation}
From (\ref{eqn-mu(N)-Zipf-conversely}) and (\ref{eqn-wt{C}_ITS}), we have $\wt{C}_{\rm ITS} = O(1)$.

Consequently, the naive ITS algorithm is expected to achieve high performance for monotone target distributions. Of course, the target distribution $\{\pi(i)\}$ is not in general monotone. In such a general case, we can sort the target distribution $\{\pi(i)\}$ by an appropriate sorting algorithm, e.g., {\it heap sort}, though this preprocessing takes $O(N\ln N)$ time. Note that, in $O(N\ln N)$ time, our BS algorithm generates $O(N)$ samples because the time complexities of the BBS and FBS procedures are $O(N)$ and $O(\ln N)$, respectively (see Table~\ref{table-BS}). Thus, the combination of the naive ITS algorithm and sorting is not competitive to our BS algorithm.

\subsection{Binary-search ITS}

Instead of sorting, there is a technique that reduces the running time of generating a sample by the ITS method; more specifically, that efficiently performs mapping a uniform random number to an element of the support set of the target distribution.  As mentioned in the introduction, such an efficient mapping is achieved by binary search. We refer to the combination of the ITS method and binary search as the binary-search ITS method. This binary-search ITS method is realized as some algorithms depending on what type of binary tree is constructed for the procedure of mapping. The standard construction of such binary trees is described in Procedure~\ref{proc-binary-tree}.

\begin{hmproc}{\bf (Standard construction of a binary tree for binary-search ITS)}\label{proc-binary-tree}
\hfill \smallskip

\noindent
{
{\bf Input:} Target distribution $\{\pi(i);i\in\bbZ_N\}$\\
{\bf Output:} Complete binary tree $\calT$

\smallskip

\noindent
Construct a complete binary tree from the $(2n+1)$ labels in $\bbZ_{2N}$ such that

\vspace{-1mm}

\begin{enumerate}
\setlength{\parskip}{0cm}
\setlength{\itemsep}{0cm}
\item the root of this tree is labeled with zero;
\item for $k=1,2,\dots,N-1$, an internal node (not the root or a leaf) with label $k$ has a parent with label $\lfloor (k-1)/2 \rfloor$, and has children with labels $2k+1$ and $2k+2$, where a left child has a smaller label than its paired right child;
\item for $i \in \bbZ_N$, the probability $\pi(i)$ is assigned to node with label $N+i$; and
\item each of all the nodes, except the leaves, stores the sum of the probabilities assigned to the leaves visited before the present node in the inorder traversal.
\end{enumerate}
%
}
\end{hmproc}

It should be noted that, although the size of the target distribution $\{\pi(i);i\in\bbZ_N\}$ is equal to $N+1$, Procedure~\ref{proc-binary-tree} constructs a complete binary tree with $2N+1$ nodes. The last $N+1$ nodes (which are all leaves) correspond to the elements of the support set $\bbZ_N$ of the target distribution $\{\pi(i);i\in\bbZ_N\}$, and, for each $i\in\bbZ_N$, node $N+i$ (node with label $N+i$) stores the probability $\pi(i)$. On the other hand, the first $N$ nodes are the root and internal nodes, and each of them stores the sum of the probabilities $\pi(i)$'s retrieved from the nodes visited according to the inorder traversal. Fig.~\ref{fig-binary_tree} provides a simple example of complete binary trees for the binary-search ITS method, where the visiting order of the nodes is $\{7, 3, 8, 1, 4, 0, 5, 2, 6\}$.

\begin{figure*}[htb]
	\centering
	\includegraphics[scale=0.4,bb=0 0 471 284]{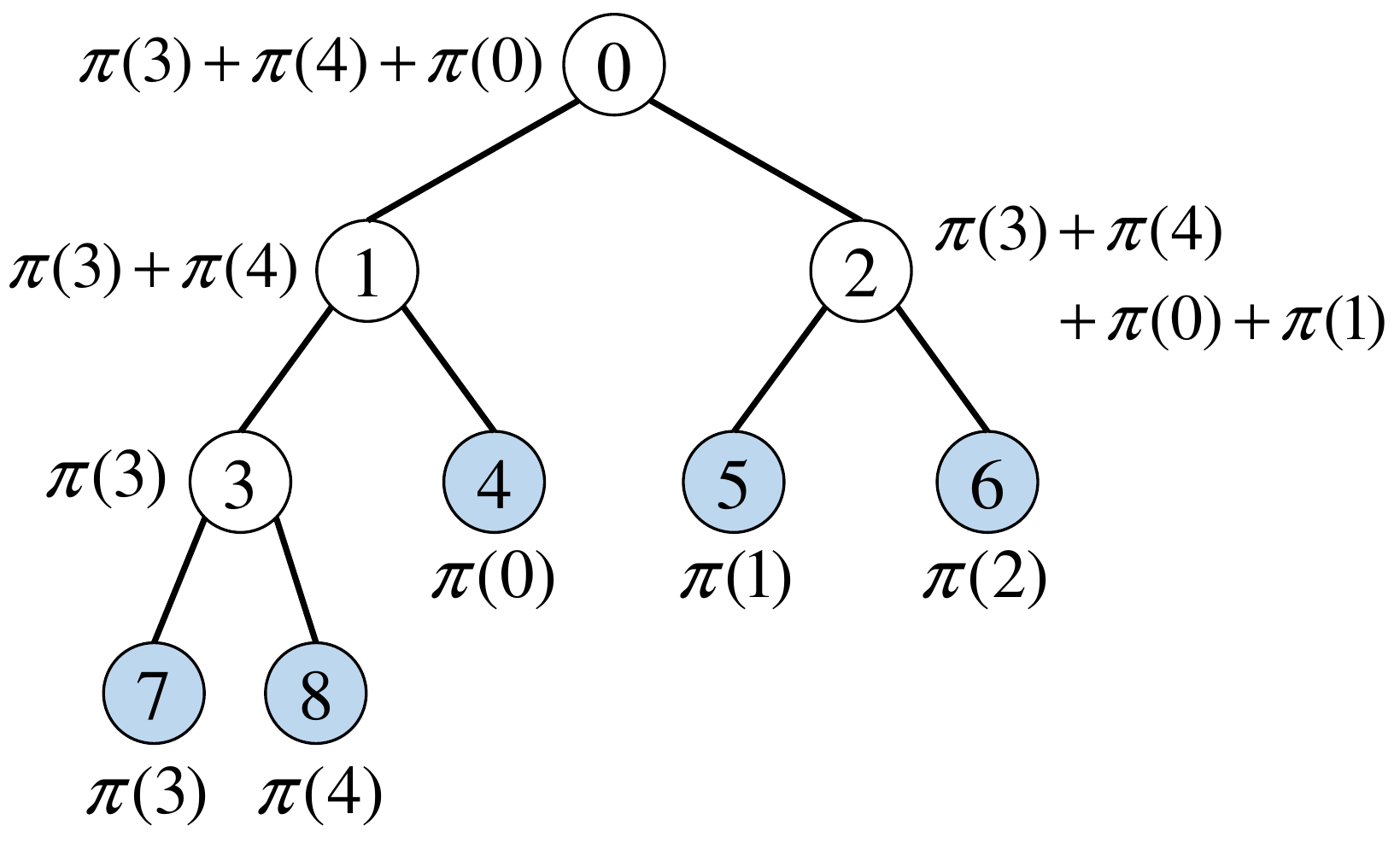}
	\caption{Example of complete binary trees for the binary-search ITS method ($N = 4$)}
	\label{fig-binary_tree}
\end{figure*}

Algorithm~\ref{algo-binary-search ITS} below describes the standard binary-search ITS algorithm based on Procedure~\ref{proc-binary-tree}. 

\begin{hmalgo}[Standard binary-search ITS]\label{algo-binary-search ITS}
\hfill \smallskip

\noindent
{
{\bf Input:}  Complete binary tree $\calT$ from Procedure~\ref{proc-binary-tree}
\\
{\bf Output:} Sample $i_{\ast}$ from $\{\pi(i);i\in\bbZ_N\}$

\vspace{-1mm}

\begin{enumerate}
\setlength{\parskip}{0cm}
\setlength{\itemsep}{0cm}
\item Generate a uniform random number $u$ in $(0,1)$, and then repeat the following operation, starting from the root and ending at one of the leaves.
\begin{enumerate}
\item If $u$ is not greater than equal to the probability of the current node, move to its left child;
\item otherwise move to its right child.
\end{enumerate}
\item Return $k-N \in \bbZ_N$, where $k$ is the label of the leaf arrived through Step~(i).
\end{enumerate}
}
\end{hmalgo}

\medskip

We consider the performance of the standard binary-search ITS algorithm, which is Algorithm~\ref{algo-binary-search ITS} together with Procedure~\ref{proc-binary-tree}. Procedure~\ref{proc-binary-tree} is the preprocessing of Algorithm~\ref{algo-binary-search ITS}, and this procedure adds, one by one, the probabilities $\pi(i)$'s retrieved according to the inorder traversal. Thus, the procedure has $O(N)$ time and space complexities. The procedure also causes $O(N)$ relative rounding error, which influences the accuracy of samples generated by Algorithm~\ref{algo-binary-search ITS}.
Algorithm~\ref{algo-binary-search ITS}, as well as Procedure~\ref{proc-binary-tree}, needs $O(N)$ space to keep the complete binary tree $\calT$. The time complexity of Algorithm~\ref{algo-binary-search ITS} is $O(\ln N)$ time complexity, because the complete binary tree $\calT$ has depth $\log_2 \lceil (N+1) \rceil$. 
As a result, the performance of the standard binary-search ITS algorithm is summarized in Table~\ref{table-binary-search ITS}.

\begin{table*}[htb]
\centering
\caption{Performance of the standard binary-search ITS algorithm}\label{table-binary-search ITS}
\begin{tabular}{|r||c|c|}
\hline
				 & \multicolumn{2}{c|}{Standard binary-search ITS}	
\\
\cline{2-3}
				 &	Procedure~\ref{proc-binary-tree} & Algorithm~\ref{algo-binary-search ITS}
\\
\hline
\hline
Time complexity  &  $O(N)$ & $O(\ln N)$
\\
\hline
Space complexity &  \multicolumn{2}{c|}{$O(N)$}
\\
\hline
Relative rounding error   &  \multicolumn{2}{c|}{$O(N)$} \\ 
\hline
\end{tabular}
\end{table*}

Tables~\ref{table-BS} and \ref{table-binary-search ITS} show that our BS algorithm has time and space complexities of the same order as those of the standard binary-search ITS algorithm. In the two algorithms, the most costly parts are their preprocessing. However, the preprocessing of our BS algorithm (i.e., Procedure~\ref{proc-BBS}) generates a sample, and it is parallelizable and thus scalable. These features do not appear in the standard binary-search ITS algorithm. In addition, our BS algorithm has a significant advantage over the standard binary-search ITS algorithm in terms of relative rounding error.

\begin{rem}
It is stated in \citealt[Section III.2]{Devr86} that Huffman tree is {\it optimal} for the binary-search ITS method in the sense that Huffman tree minimizes the average running time of mapping a uniform random number to an element of the support set of the target distribution $\{\pi(i);i\in\bbZ_N\}$. In fact, the binary-search ITS method using Huffman tree performs such mapping in $O(\log_2[1 + \mu(N)])$ time, where $\mu(N)$ is the mean of the target distribution $\{\pi(i)\}$ (for details, see \citealt[Section III.2, Theorem 2.1]{Devr86}). Unfortunately, we need $O(N\ln N)$ time to construct Huffman tree for the binary-search ITS method. Therefore, the binary search by Huffman tree, as well as, the sorting of the target distribution, is not a good strategy for the improvement of the ITS method.
\end{rem}

\section{Adaptability to multidimensional distributions}\label{sec-applicability}

In this section, we discuss the adaptability of our BS algorithm to multidimensional target distributions. Let $\bbF$ denote
\[
\bbF = \bbZ_{M_1} \times \bbZ_{M_2} \times \cdots \times \bbZ_{M_K},
\]
where $K$ is a positive integer and $M_k$'s, $k=1,2,\dots,K$, are nonnegative integers. Let $\vc{m}:=(m_1,m_2,\cdots,m_K)$ denote a vector in $\bbF$. We then define $\{p(\vc{m}); \vc{m} \in\bbF\}$ as a $K$-dimensional target distribution.

To draw a sample from this $K$-dimensional target distribution $\{p(\vc{m}); \vc{m} \in\bbF\}$, we transform it into an one-dimensional distribution $\{\pi(i);i\in\bbZ_N\}$ such that
\begin{eqnarray*}
N &=& \prod_{k=1}^K (M_k+1),
\\
\pi(f(\vc{m})) &=& p(\vc{m}),\qquad \vc{m} \in \bbF,
\end{eqnarray*}
where
\[
f(\vc{m}) = \sum_{k=1}^K  m_k \prod_{\ell=1}^{k-1} (M_{\ell}+1).
\]
In this setting, we can obtain samples $i^{(\nu)}$'s, $\nu=1,2,\dots$,  from the transformed target distribution $\{\pi(i)\}$ by the BS algorithm. We then convert the obtained samples $i^{(\nu)}$'s to $K$-dimensional vectors $\vc{m}^{(\nu)}$'s satisfying 
\[
i^{(\nu)} = f(\vc{m}^{(\nu)}).
\]
This operation can be implemented regardless of the dimension of the target distribution. Nevertheless, the BS algorithm, as well as the ITS method, cannot escape from ``Curse of Dimensionality" because its total time complexity is $O(N)$.

In what follows, we present a brief discussion of approximate sampling by the BS algorithm, which could be a solution to ``Curse of Dimensionality" in some ``lucky" cases. 
We assume that the support set $\bbF$ of $\{p(\vc{m});\vc{m}\in\bbF\}$ is possibly infinite. We also assume that
$\{p(\vc{m});\vc{m}\in\bbF\}$ denote a probability distribution such that
\begin{equation}
p(\vc{m}) = {\wt{p}(\vc{m}) \over L},\qquad \vc{m} \in \bbF,
\label{defn-p}
\end{equation}
where $L:=\sum_{\vc{m}\in\bbF}\wt{p}(\vc{m})$ is an unknown positive constant and $\wt{p}:\bbF \to (0,1)$ is a given function such that $\wt{p}(\vc{m})$ is easily calculated for all $\vc{m} \in \bbF$. For any finite $\wt{\bbF} \subseteq \bbF$, we define $\{q(\vc{m});\vc{m} \in \wt{\bbF}\}$ as a finite discrete distribution such that
\begin{equation}
q(\vc{m}) = {\wt{p}(\vc{m}) \over \wt{L}},\qquad \vc{m} \in \wt{\bbF},
\label{defn-p_{theta}}
\end{equation}
where
\begin{equation}
\wt{L} = \sum_{\vc{m} \in \wt{\bbF}} \wt{p}(\vc{m}) \le L.
\label{defn-L(theta)}
\end{equation}

According to Remark~\ref{rem-backward-BS-02}, we can draw samples from the finite distribution $\{q(\vc{m});\vc{m} \in \wt{\bbF}\}$ by applying the BS algorithm to $\{\wt{p}(\vc{m});\vc{m}\in\wt{\bbF}\}$. Note that the distribution $\{q(\vc{m});\vc{m} \in \wt{\bbF}\}$ can be considered an approximation to the distribution $\{p(\vc{m});\vc{m}\in\bbF\}$. Therefore, we can say that the samples from the distribution $\{q(\vc{m})\}$ are approximations of those from the distribution $\{p(\vc{m})\}$.

To evaluate this approximate sampling, we estimate the total variation distance between $\{p(\vc{m});\vc{m}\in\bbF\}$ and $\{q(\vc{m});\vc{m} \in \wt{\bbF}\}$, denoted by $\delta(p,q)$, i.e.,
\begin{equation*}
\delta(p,q) 
= \sum_{\vc{m}\in\wt{\bbF}} | p(\vc{m}) - q(\vc{m})|
+ \sum_{\vc{m}\in\bbF\setminus\wt{\bbF}} p(\vc{m}).
\end{equation*}
Substituting (\ref{defn-p}), (\ref{defn-p_{theta}}) and (\ref{defn-L(theta)})  into the above equation yields
\begin{eqnarray*}
\delta(p,q) 
&=& \sum_{\vc{m}\in\wt{\bbF}} \wt{p}(\vc{m}) 
\left( 
{1 \over \wt{L}} - {1 \over L}
\right)
+ \sum_{\vc{m}\in\bbF\setminus\wt{\bbF}} {\wt{p}(\vc{m}) \over L}
\nonumber
\\
&=& 1 - \sum_{\vc{m}\in\wt{\bbF}} {\wt{p}(\vc{m}) \over L}
+ \sum_{\vc{m}\in\bbF\setminus\wt{\bbF}} {\wt{p}(\vc{m}) \over L}
\nonumber
\\
&=& 2\sum_{\vc{m}\in\bbF\setminus\wt{\bbF}} {\wt{p}(\vc{m}) \over L}
\le 2\sum_{\vc{m}\in\bbF\setminus\wt{\bbF}} {\wt{p}(\vc{m}) \over \wt{L}}.
\end{eqnarray*}
Note that $\wt{L}$ is computed by the BBS procedure (Procedure~\ref{proc-BBS}). Thus, we can obtain an upper bound for $\delta(p,q)$ if we can estimate $\sum_{x\in\bbF\setminus\wt{\bbF}} \wt{p}(\vc{m})$.

We now define $\bbF(\varepsilon)$, $\varepsilon \in (0,1)$, as
\begin{equation*}
\bbF(\varepsilon)
=\left\{
\wt{\bbF} \subseteq \bbF: \sum_{\vc{m}\in\bbF\setminus \wt{\bbF}} 
{\wt{p}(\vc{m}) \over \wt{L}} < {\varepsilon \over 2}
\right\}.
\end{equation*}
If we find an $\bbF(\varepsilon)$ containing a small number of elements for a sufficiently small $\varepsilon \in (0,1)$, then we can perform 
approximate sampling from the distribution $\{p(\vc{m});\vc{m}\in\bbF\}$ with high accuracy and efficiency.

\section*{Acknowledgments}
The author acknowledges stimulating discussions with Kousei Sakaguchi. 
This research was supported in part by JSPS KAKENHI Grant Number JP15K00034.

\appendix

\makeatletter
    \renewcommand{\theequation}{%
    \thesection.\arabic{equation}}
    \@addtoreset{equation}{section}
  \makeatother

\section{Appendix: Proof of Theorem~\ref{thm-BBS}}
We first show that $\bbA$ has only one element when Steps (i) and (ii) (of Procedure~\ref{proc-BBS}) are completed. To facilitate the discussion, let $\bbA_{d+1}$ denote the set $\bbA$ before Step~(ii) starts, i.e.,
\begin{equation}
\bbA_{d+1}
= \bbB_+^{(d)}
\subseteq \bbB^d.
\label{eqn-A_{d+1}}
\end{equation}
For $j=d,d-1,\dots,1$, let $\bbA_j$ denote the set $\bbA$ at the end of the iteration of Step~(ii) with $\ell=j$. Thus, $\bbA_1$ denotes the set $\bbA$ after Step~(ii) is completed (see Fig.~\ref{fig-BBS}).
Furthermore, let $\card(\cdot)$ denote the cardinality of the set between the parentheses. It then follows from (\ref{eqn-A_{d+1}}) and Procedure~\ref{proc-BBS} that 
\begin{eqnarray*}
1 \le \card(\bbA_{d+1}) 
&\le& \card(\bbB^d) = 2^d,
\\
\card( \bbA_j )
&=& \lceil \card(\bbA_{j+1}) / 2\rceil,
~~ j=d,d-1,\dots,1.
\end{eqnarray*}
Therefore, 
\[
1 \le \card( \bbA_j ) \le 2^{j-1},
\qquad j=d,d-1,\dots,1,
\]
which leads to $\card( \bbA_1 ) = 1$. 

Next we show that the BBS procedure generates a desired sample following the target distribution $\{\pi(i)\}$. According to Step (ii-b), we have, for $\ell=d,d-1,\dots,1$ and $\vc{n}^{(\ell-1)} \in \bbB_+^{(\ell-1)}$,
\begin{eqnarray}
\lefteqn{
\PP(\{(\vc{n}^{(\ell-1)},n_{\ell},\dots,n_d) \in \bbA_{\ell}\})
}
\quad &&
\nonumber\\
&=&
\prod_{j=\ell}^d
\{ \rho_{j-1}(\vc{n}^{(j-1)} ) \}^{n_j}  
\{ \ol{\rho}_{j-1}(\vc{n}^{(j-1)} ) \}^{1 - n_j}. \qquad
\label{eqn-170216-01}
\end{eqnarray}
Combining (\ref{eqn-varpi(n)-02}), (\ref{eqn-170216-01}) and $\card( \bbA_1 ) = 1$ yields
\begin{eqnarray*}
\PP( \{ \bbA_1 = \{ \vc{n}^{(d)} \} \} )
&=& \pi(\sigma_d(\vc{n}^{(d)}))
= \pi(i_{\ast}),
\end{eqnarray*}
where $i_{\ast} = \sigma_d(\vc{n}^{(d)})$.
Recall here that Step (iii) returns $i_{\ast} = \sigma_d(\vc{n}^{(d)})$ with the unique element $\vc{n}^{(d)}$ of $\bbA_1$. Therefore,
\begin{eqnarray*}
\PP(\{\mbox{The BBS procedure returns $i_{\ast}$} \})
=  \pi(i_{\ast}).
\end{eqnarray*}
The proof is completed. 


\section*{Acknowledgments}
The author acknowledges stimulating discussions with Kousei Sakaguchi.



%
%
\bibliographystyle{elsarticle-harv} 
%
%


\end{document}